\documentclass{article}
\usepackage[dvipdfmx]{graphicx}
\usepackage{amsmath}
\usepackage{amsthm}
\usepackage{amsfonts}
\usepackage{amssymb}
\usepackage{latexsym}
\usepackage{authblk}
\def\qed{\hfill $\Box$}
\makeatletter
\renewenvironment{proof}[1][\proofname]{\par
  \normalfont
  \topsep6\p@\@plus6\p@ \trivlist
  \item[\hskip\labelsep{\bfseries #1}\@addpunct{\bfseries.}]\ignorespaces
}{
  \endtrivlist
}
\renewcommand{\proofname}{Proof}
\makeatother
\begin{document}
\title{The Dantzig selector for diffusion processes with covariates}
\author{Kou Fujimori}
\author{Yoichi Nishiyama}
\affil{Waseda University}
\date{}
\maketitle
\theoremstyle{plain}
\newtheorem{defi}{Definition}[section]
\newtheorem{lem}[defi]{Lemma}
\newtheorem{thm}[defi]{Theorem}
\newtheorem{rem}[defi]{Remark}
\newtheorem{fact}[defi]{Fact}
\newtheorem{ex}[defi]{Example}
\newtheorem{prop}[defi]{Proposition}
\newtheorem{cor}[defi]{Corollary}
\newtheorem{condition}[defi]{Condition}
\newtheorem{assumption}[defi]{Assumption}
\newcommand{\sign}{\mathop{\rm sign}}
\newcommand{\conv}{\mathop{\rm conv}}
\newcommand{\argmax}{\mathop{\rm arg~max}\limits}
\newcommand{\argmin}{\mathop{\rm arg~min}\limits}
\newcommand{\argsup}{\mathop{\rm arg~sup}\limits}
\newcommand{\arginf}{\mathop{\rm arg~inf}\limits}
\begin{abstract}
The Dantzig selector for a special parametric model of diffusion processes is studied in this paper. 
In our model, the diffusion coefficient is given as the exponential of the linear combination of 
other processes which are regarded as covariates. 
We propose an estimation procedure which is an adaptation of the Dantzig selector for 
linear regression models and prove the 
$l_q$ consistency of the estimator for all $q \in [1,\infty]$.
\end{abstract}
\section{Introduction}
The purpose of this paper is to discuss a parametric estimation problem in a high dimensional and 
sparse setting 
for a special parametric model of diffusion processes. 
We consider the stochastic process which is a solution to the stochastic differential equation given by
\begin{equation}\label{model}
X_t = X_0 + \int_0^t b(X_s) ds + \int_0^t \exp(\theta^T Z_s) dW_s, 
\end{equation}
where $\{W_t\}_{t \geq 0}$ is a standard Brownian motion, $b(\cdot)$ is a nuisance drift function, 
$\{Z_t\}_{t \geq 0} = \{(Z_t^1,Z_t^2,\ldots,Z_t^p)\}_{t \geq 0}$ is a uniformly bounded $p$ dimensional 
continuous process, which is regarded as a covariate vector, and 
$\theta$ is an unknown parameter of interest. 
We observe the process $\{X_t\}_{t \geq 0}$ at $n+1$ equidistant time points 
$0=:t_0^n < t_1^n < \ldots < t_n^n :=1$, where $t_k^n = k/n$ for $k = 0,1,\ldots,n$. 
Assume that $p = p_n \gg n$ and the number of non-zero components $S$ in the true value 
$\theta_0$ is relatively small. 
In this high dimensional and sparse setting, we consider the estimation problem of $\theta_0$. 
The covariate processes $\{Z_t^i\}_{t \geq 0}$, $i = 1, 2, \ldots, p_n$, are, for example, 
some functionals $\{\phi_i(X_t^i)\}_{t \geq 0}$ of solutions to other stochastic differential equations 
$\{X_t^i\}_{t \geq 0}$, where $\phi_i$'s are uniformly bounded smooth functions or random variables 
which do not depend on $t$. 

The parametric estimation problems in the high dimensional and sparse setting for various models 
have been investigated in contemporary statistics. 
For example, the regularized methods such as LASSO (Tibshirani (1996), Tibshirani (1997), Huang {\it et al}.\ (2013), Bradic {\it et al}.\ (2011), among others) and SCAD (Fan and Li (2002) and Bradic {\it et al}.\ (2011)) for 
regression models including Cox's proportional hazards model are studied by many researchers. 
The application of a relatively new method called the Dantzig selector, 
which is proposed by Candes and Tao (2007) for 
the estimation procedure of linear regression models, is also studied by Antoniadis {\it et al.}\ (2010) and 
Fujimori and Nishiyama (2017) for Cox's proportional hazards model. 
We will apply the Dantzig selector to our newly proposed model 
$(\ref{model})$ for which the procedure works well, 
and prove the $l_q$ consistency of the estimator for all $q \in [1,\infty]$. 
Our estimation procedure is based on the quasi-likelihood method for discretely observed data 
which has been studied intensively in low-dimensional cases, for example, 
by Yoshida (1992), Genon-Catalot and Jacod (1993), and Kessler (1997). 
Especially, we focus on the estimation problem of diffusion coefficients
for high-frequency observed data on a 
fixed time interval, which can be seen in Genon-Catalot and Jacod (1993). 
We consider the estimation problem in a high-dimensional and sparse setting, although all of their results are concerned with low-dimensional cases.

This paper is organized as follows. 
The settings for the model, some regularity conditions, and the estimation procedure 
are given in Section 2. 
In Section 3, we state our main results. 
The proofs are presented in Section 4. 
Our methods of proofs are similar to Huang {\it et al}.\ (2013) who proved the 
consistency of LASSO estimator for Cox's proportional hazards model 
and to Fujimori and Nishiyama (2017) who dealt with the Dantzig selector for the proportional hazards model. 

Throughout this paper, for every $q \in [1,\infty]$, we denote by $\|\cdot\|_q$ the $l_q$-norm of 
$p$ dimensional vector, which is defined as follows:
\begin{align*}
\|v\|_q &= \left(\sum \limits_{j=1}^p |v_j|^q \right)^{\frac{1}{q}},\quad q < \infty ; \\
\|v\|_{\infty} &= \sup \limits_{1 \leq j \leq p} |v_j|.
\end{align*}
Moreover, for a $m \times n$ matrix $A$, where $m,\ n \in \mathbb{N}$, we define $\|A\|_{\infty}$ by
\[\|A\|_{\infty} := \sup \limits_{1 \leq i \leq m} \sup \limits_{1 \leq j \leq n} |A_i^j|,\]
where $A_i^j$ denotes the $(i,j)$-component of the matrix $A$.
\section{Preliminaries}
Let $\{W_t\}_{t \geq 0}$ be a standard Brownian motion defined on a probability space 
$(\Omega,\mathcal{F},P)$, and $\{Z_t\}_{t \geq 0} := \{(Z_t^1,Z_t^2,\ldots,Z_t^p)\}_{t \geq 0}$
be a uniformly bounded $p$ dimensional continuous process. 
We introduce the filtration $\{\mathcal{F}_t\}_{t \geq 0}$ defined by  
\[\mathcal{F}_t := \mathcal{F}_0 \lor \sigma(W_s, Z_s : s \in [0, t]), \quad t \geq 0,\]
where $\mathcal{F}_0$ is a $\sigma$-field independent of $\{W_t\}_{t \geq 0}$, and $\{Z_t\}_{t \geq 0}$. 
Let us consider the $1$ dimensional stochastic differential equation (\ref{model}): 
\[X_t = X_0 + \int_0^t b(X_s) ds + \int_0^t \exp(\theta^T Z_s) dW_s,\]
where $x \mapsto b(x)$ is a nuisance drift function which satisfies appropriate regularity conditions 
presented later, 
and $\theta \in \mathbb{R}^{p_n}$ is an unknown parameter of interest. 
We observe the process $\{X_t\}_{t \geq 0}$ at $n+1$ discrete time points 
$0 =:t_0^n < t_1^n < t_2^n < \cdots < t_n^n:=1$, where $t_k^n := k/n$. 
Assume that $p = p_n \gg n$ and the number of non-zero components $S$ in the true value
$\theta_0$ is a fixed constant. 
In this high dimensional and sparse setting, we consider the estimation problem of $\theta_0$.
The quasi-likelihood function $L_n(b;\theta)$ is given by 
\[L_n(b;\theta) =\prod_{k=1}^n \frac{1}{\sqrt{2\pi \exp(2\theta^T Z_{t_{k-1}^n})\Delta_n}} 
\exp \left(- \frac{|X_{t_k^n}-X_{t_{k-1}^n}-b(X_{t_{k-1}^n})\Delta_n|^2}{2 \exp(2 \theta^T Z_{t_{k-1}^n}) \Delta_n}\right),\]
where $\Delta_n := t_k^n - t_{k-1}^n = 1/n$.
Put $l_n(b;\theta) := \log L_n(b;\theta)$, and define the $\mathbb{R}^{p_n}$-valued function 
$\psi_n(b;\theta) = (\psi_n^1(b;\theta), \psi_n^2(b;\theta),\ldots,\psi_n^{p_n}(b;\theta))$ by 
\begin{eqnarray*}
\psi_n(b;\theta) 
&:=& \frac{1}{n} \dot {l}_n(b;\theta) \\
&=& \frac{1}{n \Delta_n} \sum_{k=1}^n Z_{t_{k-1}^n} \exp(-2 \theta^T Z_{t_{k-1}^n})
|X_{t_k^n} - X_{t_{k-1}^n}-b(X_{t_{k-1}^n}) \Delta_n|^2 \\
&& \qquad \qquad 
- Z_{t_{k-1}^n}\Delta_n. 
\end{eqnarray*}
Moreover, we define the $p_n \times p_n$ matrix-valued function $V_n(b;\theta)$ by 
\begin{align*}
V_n(b;\theta)
&:= -\frac{1}{n} \ddot l_n(b;\theta) \\
&=  \frac{2}{n \Delta_n} \sum_{k=1}^n Z_{t_{k-1}^n} Z_{t_{k-1}^n}^T \exp(-2 \theta^T Z_{t_{k-1}^n})
|X_{t_k^n} - X_{t_{k-1}^n}-b(X_{t_{k-1}^n}) \Delta_n|^2.
\end{align*}
Note that $V_n(b;\theta)$ is a nonnegative definite matrix. 
Hereafter, we assume the following conditions.
\begin{assumption}\label{regularity}
\begin{description}
\item[$(i)$]
There exists a constant $\tilde{L} > 0$ such that for all $x,y \in \mathbb{R}$, 
\[|b(x) - b(y)| \leq \tilde{L} |x-y|.\]
\item[$(ii)$]
There exists a constant $C>0$ such that
\[\sup_{t \in [0,1]} \sup_{1 \leq i \leq \infty} |Z_t^i| \leq C.\]
\item[$(iii)$]
For every $r \geq 1$, it holds that
\[ \sup_{t \in [0, 1]} E \left[ |X_t|^r\right] < \infty.\]
\item[$(iv)$]
For every $r \in \mathbb{N}$, there exists a constant $\tilde{C}_r$ such that for every $n \in \mathbb{N}$, $i \in \{1,2,\ldots,p_n\}$ and $k = 1,2,\ldots,n$,
\[E\left[\sup_{s \in [{t_{k-1}^n}, {t_k^n}]} | X_s - X_{{t_{k-1}^n}}|^r \right] \leq \tilde{C}_r \Delta_n^{\frac{r}{2}},\]
\[E\left[\sup_{s \in [{t_{k-1}^n}, {t_k^n}]} | Z^i_s - Z^i_{{t_{k-1}^n}}|^r \right] \leq \tilde{C}_r \Delta_n^{\frac{r}{2}}.\]
\end{description}
\end{assumption}
Assumption $(iv)$ is satisfied if $Z_t^i$, $i = 1, 2, \ldots,p_n$, are appropriate transformation of stochastic processes
which are solutions to other SDEs as mentioned in Introduction. 
In Section $4$, we will show that $b(\cdot)$ can be ignored under Assumption \ref{regularity}. 
We thus define the estimator $\hat{\theta}_n$ by the Dantzig selector as
\[\hat{\theta}_n := \argmin_{\theta \in \mathcal{C}_n} \|\theta\|_1, \quad
\mathcal{C}_n := \{\theta \in \mathbb{R}^{p_n} : \|\psi_n(0;\theta)\|_{\infty} \leq \gamma\},
\]
where $\gamma$ is a tuning parameter by setting $b = 0$. 

Define the $p_n \times p_n$ matrix $J_n$ by 
\[J_n := \frac{2}{n} \sum_{k=1}^n Z_{t_{k-1}^n}Z_{t_{k-1}^n}^T,\]
which will be proved to approximate $V_n(0;\theta_0)$ in Section $4$. 
We introduce the following factors $(A),\ (B)$ and $(C)$ in order to prove the consistency of 
the estimator $\hat{\theta}_n$.
\begin{defi}\label{factors}
For every index set $T \subset \{1,\ 2,\ \cdots,\ p_n\}$ and $h \in \mathbb{R}^{p_n}$, 
$h_T$ is a $\mathbb{R}^{|T|}$ dimensional sub-vector of $h$ constructed by extracting the components of $h$ corresponding to the indices in $T$. Define the set $C_T$ by 
\[C_T := \{h \in \mathbb{R}^{p_n} : \|h_{T^c}\|_1 \leq \|h_{T}\|_1\} .\]
We introduce the following three factors.
\begin{description}
\item[$(A)$ Compatibility factor]
\[\kappa(T_0;J_n) := \inf \limits_{0 \not = h \in C_{T_0}} \frac{S^{\frac{1}{2}} (h^T J_n  h)^{\frac{1}{2}}}{\|h_{T_0}\|_1}.\]
\item[$(B)$ Weak cone invertibility factor]
\[F_q(T_0;J_n) :=  \inf \limits_{0 \not = h \in C_{T_0}} \frac{S^{\frac{1}{q}} (h^T J_n h)^{\frac{1}{2}}}{\|h_{T_0}\|_1 \|h\|_q},\quad q \in [1,\infty),\]
\[F_\infty(T_0;J_n) := \inf \limits_{0 \not = h \in C_{T_0}} \frac{(h^T J_n h)^{\frac{1}{2}}}{\|h\|_\infty}.\]
\item[$(C)$ Restricted eigenvalue]
\[RE(T_0;J_n) :=  \inf \limits_{0 \not = h \in C_{T_0}} \frac{ (h^T J_n h)^{\frac{1}{2}}}{\|h\|_2}.\]
\end{description}
\end{defi}
We assume the next condition to derive our main results. 
\begin{assumption}\label{matrix}
For every $\epsilon>0$, there exist $\delta >0$ and $n_0 \in \mathbb{N}$ 
such that for all $n \geq n_0$
\[P(\kappa(T_0 ; J_n) > \delta) \geq 1- \epsilon.\]
\end{assumption}
Noting that $\|h_{T_0}\|_1^q \geq \|h_{T_0}\|_q^q$ for all $q \geq 1$, 
we can see that 
$\kappa(T_0;J_n)  \leq 2\sqrt{S} RE(T_0;J_n)$, and 
$\kappa(T_0;J_n) \leq F_q(T_0;J_n)$. 
So under Assumption \ref{matrix}, $RE(T_0;J_n)$ and $F_q(T_0;J_n)$ also satisfy 
the corresponding conditions.
\section{The $l_q$ consistency of the estimator}
The following theorems are our main results. 
The proofs are provided in Section $4$.
Hereafter, we assume that $\gamma_n$ and $p_n$ satisfy that 
\begin{eqnarray}
\gamma_n &=& K_0 \Delta_n^{\frac{1}{2}-\alpha}, \label{gamma}\\
\log(1+p_n) &=& O(n^\zeta), \label{p_n}
\end{eqnarray}
where $K_0>0$, $0 < \alpha < 1/2$, $0 < \zeta < 2\alpha$ are some constants.
\begin{thm}\label{re}
Suppose that $\gamma_n$ and $p_n$ satisfy (\ref{gamma}) and (\ref{p_n}) respectively.
Under Assumptions \ref{regularity} and \ref{matrix}, the following (i) and (ii) hold true 
for some positive constants $K_2$ and $K_3$.
\begin{description}
\item{$(i)$}
It holds that
\[\lim_{n \rightarrow \infty} 
P \left(\|\hat{\theta}_n - \theta_0\|_2^2 \geq 
\frac{K_2 {\gamma_n} + K_3 \epsilon_n}{RE^2(T_0;J_n)} \right) = 0.\]
In particular, it holds that $\|\hat{\theta}_n - \theta_0\|_2 \rightarrow^p 0$.
\item{$(ii)$}
It holds that
\[\lim_{n \rightarrow \infty} 
P \left(\|\hat{\theta}_n - \theta_0\|_\infty^2 \geq 
\frac{K_2 {\gamma_n} + K_3 \epsilon_n}{F_\infty^2(T_0;J_n)} \right) = 0.\]
In particular, it holds that $\|\hat{\theta}_n - \theta_0\|_\infty \rightarrow^p 0$.
\end{description}
\end{thm}
\begin{thm}\label{comp}
Under the same assumption as Theorem \ref{re}, the following $(i)$ and (ii) hold true for a positive constant 
$K_4$.
\begin{description}
\item[$(i)$]
It holds that 
\[\lim_{n \rightarrow \infty} P \left(\|\hat{\theta}_n - \theta_0\|_1 \geq 
\frac{4K_4 S \gamma_n}{\kappa^2(T_0;J_n)-4S\epsilon_n}
\right) = 0.\]
In particular, it holds that $\|\hat{\theta}_n - \theta_0\|_1 \rightarrow^p 0$.
\item[$(ii)$]
It holds for every $q \in (1,\infty)$ that
\[
\lim_{n \rightarrow \infty} P \left(\|\hat{\theta}_n - \theta_0\|_q \geq 
\xi_{n,q}
\right) = 0,
\]
where 
\[\xi_{n,q} 
:= 
\frac{2S^{\frac{1}{q}}\epsilon_n}{F_q(T_0;J_n)} \cdot \frac{2K_4 S \gamma_n}{\kappa^2(T_0;J_n)-2S\epsilon_n} +
\frac{2K_4 S^{\frac{1}{q}}\gamma_n}{F_q(T_0;J_n)}.\]
In particular, it holds for all $q \in (1,\infty)$ that $\|\hat{\theta}_n - \theta_0\|_q \rightarrow^p 0$.
\end{description}
\end{thm}

\section{Proofs}
\subsection{A stochastic inequality for the gradient of the log-quasi-likelihood}
In this subsection, we will show that under Assumption \ref{regularity},
\[\lim_{n \rightarrow \infty} P(\|\psi_n(b;\theta_0)\|_{\infty} \geq \gamma_n) =0\]
for any $b(\cdot)$ if $\gamma_n$ satisfies (\ref{gamma}) although we are interested only in the case of $b=0$.
First, let us decompose 
$\psi_n^i(b;\theta_0) = A_n^i + B_n^i + C_n^i$, where
\[
A_n^i := \frac{1}{n \Delta_n} \sum_{k=1}^n Z_{t_{k-1}^n}^i \exp(-2 \theta_0^T Z_{t_{k-1}^n}) 
\left| \int_{t_{k-1}^n}^{t_k^n} \{b(X_s) - b(X_{t_{k-1}^n})\} ds \right|^2, \]
\begin{eqnarray*}
B_n^i 
&:=& \frac{2}{n \Delta_n} \sum_{k=1}^n Z_{t_{k-1}^n}^i \exp(-2 \theta_0^T Z_{t_{k-1}^n})
\left(\int_{t_{k-1}^n}^{t_k^n} \{b(X_s)-b(X_{t_{k-1}^n})\} ds \right) \\
&& \qquad \qquad \times
\left(\int_{t_{k-1}^n}^{t_k^n} \exp(\theta_0^T Z_s) dW_s \right) 
\end{eqnarray*}
and
\[
C_n^i :=  \frac{1}{n \Delta_n} \sum_{k=1}^n Z_{t_{k-1}^n}^i \exp(-2 \theta_0^T Z_{t_{k-1}^n})
\left|\int_{t_{k-1}^n}^{t_k^n} \exp(\theta_0^T Z_s) dW_s \right|^2 - Z_{t_{k-1}^n}^i\Delta_n.
\]
We further decompose $C_n^i = D_n^i + E_n^i$, where
\begin{eqnarray*}
D_n^i 
&:=&  \frac{1}{n \Delta_n} \sum_{k=1}^n Z_{t_{k-1}^n}^i \exp(-2 \theta_0^T Z_{t_{k-1}^n})
\left|\int_{t_{k-1}^n}^{t_k^n} \exp(\theta_0^T Z_s) dW_s \right|^2 \\
&& \qquad \qquad
-Z_{t_{k-1}^n}^i(W_{t_k^n}-W_{t_{k-1}^n})^2 
\end{eqnarray*}
and
\[
E_n^i := \frac{1}{n \Delta_n} \sum_{k=1}^n Z_{t_{k-1}^n}^i \{(W_{t_k^n} - W_{t_{k-1}^n})^2 - \Delta_n\}.
\]
\begin{lem} \label{A}
If $\gamma_n$ satisfies (\ref{gamma}),
then it holds that
\[\lim_{n \rightarrow \infty} P\left(\sup_{1 \leq i \leq p_n} |A_n^i| \geq \gamma_n\right) =0\]
for any $b(\cdot)$ which satisfies Assumption \ref{regularity}.
\end{lem}
\begin{proof}
It follows from Markov's inequality and Schwartz's inequality and Assumption \ref{regularity} that 
\begin{align*}
P \left(\sup_{1 \leq i \leq p_n} |A_n^i| \geq \gamma_n \right) 
&\leq \frac{C\exp(2C \|\theta_0\|_1)}{n \Delta_n \gamma_n} \sum_{k = 1}^n 
E \left[\left|\int_{t_{k-1}^n}^{t_k^n} \{b(X_s) - b(X_{t_{k-1}^n})\} ds\right|^2 \right] \\
&\leq \frac{C\exp(2C \|\theta_0\|_1)}{n \Delta_n \gamma_n} \sum_{k = 1}^n 
E \left[\left(\int_{t_{k-1}^n}^{t_k^n} \{b(X_s) - b(X_{t_{k-1}^n})\}^2 ds\right) \Delta_n\right] \\
&\leq \frac{C\exp(2C \|\theta_0\|_1)}{n \Delta_n \gamma_n} \sum_{k = 1}^n
E \left[\Delta_n\int_{t_{k-1}^n}^{t_k^n} \tilde{L}^2 |X_s - X_{t_{k-1}^n}|^2 ds\right] \\
&\leq \frac{C\exp(2C \|\theta_0\|_1)}{n \gamma_n} \sum_{k = 1}^n
\tilde{L}^2 \int_{t_{k-1}^n}^{t_k^n} E[|X_s - X_{t_{k-1}^n}|^2]ds \\
&\leq \frac{C\exp(2C \|\theta_0\|_1)}{\gamma_n}
\tilde{L}^2 \tilde{C_2} \Delta_n^2. \\ 
\end{align*}
Noting that $\Delta_n \rightarrow 0$ and $\gamma_n = K_0 \Delta_n^{\frac{1}{2}-\alpha}$, 
we obtain the conclusion. 
\qed
\end{proof}
\begin{lem}\label{B}
Under the same assumptions as Lemma \ref{A}, it holds that
\[\lim_{n \rightarrow \infty} P\left(\sup_{1 \leq i \leq p_n} |B_n^i| \geq \gamma_n \right) = 0.\]
for any $b(\cdot)$ which satisfies Assumption \ref{regularity}.
\end{lem}
\begin{proof}
Using Markov's inequality and Schwartz's inequality, we have that 
\begin{eqnarray*}
\lefteqn{P \left(\sup_{1 \leq i \leq p_n} |B_n^i| \geq \gamma_n \right)} \\
&\leq& \frac{2C\exp(2C\|\theta_0\|_1)}{n \Delta_n \gamma_n} \sum_{k=1}^n 
\left(E \left[\left|\int_{t_{k-1}^n}^{t_k^n} \{b(X_s)-b(X_{t_{k-1}^n})\} ds\right|^2 \right]\right)^{\frac{1}{2}}\\
&&\qquad \quad \quad \quad \quad \quad \quad \ \ 
\times \left(E \left[\left|\int_{t_{k-1}^n}^{t_k^n} \exp(\theta_0^T Z_s) dW_s \right|^2\right]\right)^{\frac{1}{2}} \\
&\leq& \frac{2C\exp(2C\|\theta_0\|_1)}{n \Delta_n \gamma_n} \sum_{k=1}^n 
\left(E\left[\Delta_n \int_{t_{k-1}^n}^{t_k^n} |b(X_s)-b(X_{t_{k-1}^n})|^2 ds\right]\right)^{\frac{1}{2}}\\
&&\qquad \quad \quad \quad \quad \quad \quad \ \ 
\times \left(E\left[\int_{t_{k-1}^n}^{t_k^n} \exp(2\theta_0^T Z_s)ds\right]\right)^{\frac{1}{2}}\\
&\leq& \frac{2C\exp(2C\|\theta_0\|_1)}{n \Delta_n \gamma_n} n 
\left(\tilde{L}^2\tilde{C}_2\Delta_n^3\right)^{\frac{1}{2}}
\left(\exp(2C\|\theta_0\|_1) \Delta_n\right)^{\frac{1}{2}}\\
&\leq& \frac{C\tilde{L}\tilde{C}_2^{\frac{1}{2}}\Delta_n\exp(3C\|\theta_0\|_1)}{\gamma_n}.
\end{eqnarray*}
The right-hand side of this inequality tends to $0$ as $n \rightarrow \infty$. 
\qed
\end{proof}
Lemma \ref{A}, and Lemma \ref{B} imply that we can ignore the effect of $b(\cdot)$.
So we may take $b(x) = 0$ when we define the estimator $\hat{\theta}_n$. 
The following lemmas give some inequalities about $D_n^i$ and $E_n^i$.
\begin{lem}\label{D}
Under the same assumption as Lemma \ref{A}, it holds that 
\[\lim_{n\rightarrow \infty} P\left(\sup_{1 \leq i \leq p_n} |D_n^i| \geq \gamma_n \right) = 0.\]
\end{lem}
\begin{proof}
It follows from Markov's inequality and Schwartz's inequality that
\begin{eqnarray*}
P\left(\sup_{1 \leq i \leq p_n} |D_n^i| \geq \gamma_n\right)
&\leq& \frac{C}{n\Delta_n \gamma_n} \sum_{k=1}^n 
E\left[
\left|D_1\right|  
\cdot \left|D_2 \right|
\right] \\
&\leq& \frac{C}{n\Delta_n \gamma_n} \sum_{k=1}^n
(E[|D_1|^2])^{\frac{1}{2}} (E[|D_2|^2])^{\frac{1}{2}}, 
\end{eqnarray*}
where $D_1$ and $D_2$ are defined as follows
\[D_1 := \int_{t_{k-1}^n}^{t_k^n} \{\exp(\theta_0^T[Z_s-Z_{t_{k-1}^n}])+1\}dW_s,\]
\[D_2 := \int_{t_{k-1}^n}^{t_k^n} \{\exp(\theta_0^T[Z_s-Z_{t_{k-1}^n}])-1\}dW_s.\]
We can see that
\begin{align*}
(E[|D_1|^2])^{\frac{1}{2}}
&= \left(E\left[
\int_{t_{k-1}^n}^{t_k^n} 
\{\exp(\theta_0^T [Z_s-Z_{t_{k-1}^n}])+1\}^2 ds
\right]\right)^{\frac{1}{2}} \\
&\leq (\exp(2C\|\theta_0\|_1)+1) \Delta_n^{\frac{1}{2}}.
\end{align*}
Noting that there exists a positive constant $C_1$ such that 
\begin{align*}
|\exp(\theta_0^T [Z_s-Z_{t_{k-1}^n}]) -1 |
&\leq C_1 |\theta_0^T [Z_s-Z_{t_{k-1}^n}]| \\
&\leq C_1 \|\theta_0\|_1 \max_{ i \in T_0} |Z_s^i - Z_{t_{k-1}^n}^i|,
\end{align*}
where $T_0 := \{i : \theta_0^i \not = 0\}$,
we have that 
\begin{align*}
(E[|D_2|^2])^{\frac{1}{2}}
&= \left(E\left[
\int_{t_{k-1}^n}^{t_k^n} \{\exp(\theta_0^T[Z_s - Z_{t_{k-1}^n}])-1\}^2 ds
\right]\right)^{\frac{1}{2}} \\
&\leq \left(E\left[
\int_{t_{k-1}^n}^{t_k^n} C_1^2 \|\theta_0\|_1^2 \max_{i \in T_0}|Z_s^i - Z_{t_{k-1}^n}^i|^2 ds
\right]\right)^{\frac{1}{2}} \\
&\leq C_1 \tilde{C}_2 \|\theta_0\|_1 \Delta_n.
\end{align*}
Consequently, it holds that 
\[P\left(\sup_{1 \leq i \leq p_n} |D_n^i| \geq \gamma_n \right)
\leq \frac{C C_1 \tilde{C}_2 \|\theta_0\|_1 (\exp(2C\|\theta_0\|_1)+1) \Delta_n^{\frac{1}{2}}}{\gamma_n} \rightarrow 0.\]
We thus obtain the conclusion.
\qed
\end{proof}
\begin{lem}\label{E}
Suppose that $\gamma_n$ and $p_n$ satisfy (\ref{gamma}) and (\ref{p_n}) respectively.
Then, it holds that 
\[\lim_{n \rightarrow \infty} P\left(\sup_{1 \leq i \leq p_n} |E_n ^i| \geq 3\gamma_n \right) = 0.\]
\end{lem}
\begin{proof}
Put $U_{t_k^n} := |W_{t_k^n} - W_{t_{k-1}^n}|^2 - \Delta_n$ and 
$\eta := \Delta_n^{1/2 + \alpha - \beta}$, where $0 < \beta < 2\alpha - \zeta$ is a constant.
Then, we have that
\begin{align*}
E_n^i
&= \frac{1}{n \Delta_n} \sum_{k=1}^n Z_{t_{k-1}^n}^i U_{t_k^n} 1_{\{|U_{t_k^n}| \leq \eta \}}
+ \frac{1}{n \Delta_n} \sum_{k =1}^n Z_{t_{k-1}^n}^i U_{t_k^n} 1_{\{|U_{t_k^n}| > \eta \}} \\
&=: F_n^i + G_n^i.
\end{align*}
It is sufficient to prove that 
$P(\sup_i |F_n^i| \geq 2\gamma_n) \rightarrow 0$ and $P(\sup_i |G_n^i| \geq \gamma_n) \rightarrow 0$.
Note that 
\begin{eqnarray*}
F_n^i
&=& \frac{1}{n \Delta_n} \sum_{k =1}^n Z_{t_{k-1}^n}^i \{U_{t_k^n} 1_{\{|U_{t_k^n}| \leq \eta \}}
- E[U_{t_k^n}1_{\{|U_{t_k^n}| \leq \eta\}}|\mathcal{F}_{t_{k-1}^n}]\} \\
&&\qquad \qquad
+ Z_{t_{k-1}^n}^i  E[U_{t_k^n}1_{\{|U_{t_k^n}| \leq \eta\}}| \mathcal{F}_{t_{k-1}^n}] \\
&=:& H_n^i + I_n^i.
\end{eqnarray*}
We can see that for all $k$ and $i$, 
\[|Z_{t_{k-1}^n}^i\{U_{t_k^n}1_{\{|U_{t_k^n}| \leq \eta\}} - E[U_{t_k^n}1_{\{|U_{t_k^n}| \leq \eta\}}]\}| \leq 2C\eta\]
\[E[|Z_{t_{k-1}^n}^i|^2 \{U_{t_k^n}1_{\{|U_{t_k^n}| \leq \eta\}} - E[U_{t_k^n}1_{\{|U_{t_k^n}| \leq \eta\}}]| \mathcal{F}_{t_{k-1}^n}]\}^2 | \mathcal{F}_{t_{k-1}^n} ] \leq C^2 \Delta_n^2.\]
Now, it follows from Bernstein's inequality for martingales (See Theorem $1.6$ from Freedman (1975).) that
\[
P \left(|H_n^i| \geq \gamma_n \right)
\leq 2 \exp \left(- \frac{\gamma_n^2}{2(2C\eta \gamma_n + C^2\Delta_n^2)} \right).
\]
Write $\|\cdot\|_{\Phi}$ for Orlicz norm with respect to $\Phi(x) := e^x-1$.
Lemma $2.2.10$ from van der Vaart and Wellner (1996) implies that there exists a constant $L>0$
depending only on $\Phi$ such that 
\[
\left\|\sup_{1 \leq i \leq p_n} |H_n^i|\right\|_{\Phi} \leq L\left\{2C\eta \log(1+p_n) +\sqrt{C^2\Delta_n^2 \log(1+p_n)}\right\}.
\]
Using Markov's inequality, we have that 
\begin{align*}
P\left(\sup_{1 \leq i \leq p_n} |H_n^i| \geq \gamma_n \right)
&= P \left(\Phi \left(
\frac{\sup_i |H_n^i|}{\left\|\sup_i |H_n^i|\right\|_{\Phi}}\right) 
\geq \Phi \left(\frac{\gamma_n}{\left\|\sup_i |H_n^i|\right\|_{\Phi}}\right)\right) \\ 
&\leq \Phi \left(\frac{\gamma_n}{\left\|\sup_i |H_n^i|\right\|_{\Phi}}\right)^{-1} \\
&\leq \Phi
\left(\frac{\gamma_n}{L\left\{2C\eta \log(1+p_n) +\sqrt{C^2\Delta_n^2 \log(1+p_n)}\right\}}\right)^{-1} \\
&\rightarrow 0.
\end{align*}
On the other hand, it holds that 
\begin{align*}
I_n^i 
&= \frac{1}{n \Delta_n} \sum_{k = 1}^n Z_{t_{k-1}^n}^i 
\left\{E[U_{t_k^n} - U_{t_k^n} 1_{\{|U_{t_k^n}| > \eta\}} | \mathcal{F}_{t_{k-1}^n}]\right\} \\
&= \frac{1}{n \Delta_n} \sum_{k = 1}^n Z_{t_{k-1}^n}^i E[-U_{t_k^n} 1_{\{|U_{t_k^n}| > \eta\}}| \mathcal{F}_{t_{k-1}^n}].
\end{align*}
So we thus obtain that 
\begin{align*}
P \left(\sup_{1 \leq i \leq p_n} |I_n^i| \geq \gamma_n \right)
&\leq \frac{1}{\gamma_n} E\left[\sup_{1 \leq i \leq p_n} 
\left|\frac{1}{n\Delta_n} \sum_{k=1}^n Z_{t_{k-1}^n}^i 
E[U_{t_k^n}1_{\{|U_{t_k^n}| > \eta \}}| \mathcal{F}_{t_{k-1}^n}]\right|\right] \\
&\leq \frac{C}{n \Delta_n \gamma_n} \sum_{k=1}^n E\left[E\left[\frac{|U_{t_k^n}|^2}{\eta}| \mathcal{F}_{t_{k-1}^n}\right]\right] \\
&= \frac{2C \Delta_n}{\gamma_n \eta} \\
&\rightarrow 0.
\end{align*}
A similar calculation leads us that 
\[P \left(\sup_{1 \leq i \leq p_n}|G_n^i| \geq \gamma_n\right) \rightarrow 0.\]
This yields the conclusion.
\qed
\end{proof}
After all, we obtain the next lemma.
\begin{lem}\label{grad}
Suppose that $\gamma_n$ and $p_n$ satisfy (\ref{gamma}) and (\ref{p_n}) respectively.
Then, it holds for any $b(\cdot)$ that
\[\lim_{n \rightarrow \infty} P\left(\|\psi_n (b;\theta_0)\|_\infty \geq 6 \gamma_n\right) = 0.\]
\end{lem}
This lemma states that the true value $\theta_0$ belongs to the constraint set $\mathcal{C}_n$ 
with large probability when the sample size $n$ is large.
\subsection{Some discussions on the Hessian}
In this subsection, we prepare two lemmas for $V_n(0;\theta_0)$.
The next lemma states that $V_n(0;\theta_0)$ is approximated by $J_n$.
\begin{lem}\label{app}
The random sequence $\epsilon_n$ defined by  
\[\epsilon_n := \left\|V_n(0;\theta_0) - J_n \right\|_\infty \]
converges in probability to $0$.
\end{lem}
\begin{proof}
It holds that 
\begin{align*}
V_n(0;\theta_0)
&= \frac{2}{n \Delta_n} \sum_{k=1}^n Z_{t_{k-1}^n}Z_{t_{k-1}^n}^T \exp(-2\theta_0^T Z_{t_{k-1}^n}) 
\left|\int_{t_{k-1}^n}^{t_k^n} \exp(\theta_0^T Z_s) dW_s\right|^2 \\
&=(I) + (II) + (III),
\end{align*}
where 
\begin{eqnarray*}
(I)
&=& \frac{2}{n \Delta_n} \sum_{k=1}^n Z_{t_{k-1}^n}Z_{t_{k-1}^n}^T
\left|\int_{t_{k-1}^n}^{t_k^n} \exp(\theta_0^T[Z_s-Z_{t_{k-1}^n}]) dW_s\right|^2 \\
&& \qquad \qquad
-Z_{t_{k-1}^n}Z_{t_{k-1}^n}^T|W_{t_k^n}-W_{t_{k-1}^n}|^2,
\end{eqnarray*}
\[(II)
= \frac{2}{n \Delta_n} \sum_{k=1}^n Z_{t_{k-1}^n}Z_{t_{k-1}^n}^T
\left\{
|W_{t_k^n} - W_{t_{k-1}^n}|^2 - \Delta_n
\right\},
\]
and
\[(III)
= \frac{2}{n \Delta_n} \sum_{k=1}^n Z_{t_{k-1}^n}Z_{t_{k-1}^n}^T \Delta_n = J_n.
\]
Using triangle inequality, we have that
\[\left\|V_n(0;\theta_0) - J_n \right\|_\infty
\leq \|(I)\|_\infty + \|(II)\|_\infty.\]
As well as the proof of Lemma \ref{D} and Lemma \ref{E}, 
we can prove that $\|(I)\|_\infty$ and $\|(II)\|_\infty$ are $o_p(1)$.
\qed
\end{proof}
The relationship between $\psi_n(0;\hat{\theta}_n)-\psi_n(0;\theta_0)$ and $V_n(0;\theta_0)$
are provided by the lemma below.
\begin{lem}\label{pollard}
Define that $I := [-2C\|\theta_0\|_1, 2C\|\theta_0\|_1]$,
\[g(x) :=
\begin{cases}
\frac{e^{2x}-1}{x} & (x \not = 0) \\
2 & (x=0)
\end{cases}\]
and $\nu := \min_{x \in I} g(x)$.
Then, it holds for $h := \theta_0 - \hat{\theta}_n$ that 
\[\frac{\nu}{2}h^T V_n(0;\theta_0) h \leq h^T [\psi_n(0;\hat{\theta}_n) - \psi_n(0;\theta_0)].\]
\end{lem}
\begin{proof}
We have that
\begin{eqnarray*}
h^T[\psi_n(0;\hat{\theta}_n) - \psi_n(0;\theta_0)]
&=& \frac{1}{n \Delta_n} \sum_{k=1}^n h^T Z_{t_{k-1}^n} \exp(-2\theta_0^T Z_{t_{k-1}^n})|X_{t_k^n}-X_{t_{k-1}^n}|^2 \\
&& \qquad \qquad \qquad \qquad \times
\{\exp(2h^TZ_{t_{k-1}^n})-1\} 
\end{eqnarray*}
Note that $h^T Z_{t_{k-1}^n} \in I$ for all $k=1,2,\ldots,n$. 
Noting moreover that $x(e^{2x}-1) \geq \nu x^2$, we can see that 
\begin{align*}
h^T[\psi_n(0;\hat{\theta}_n) - \psi_n(0;\theta_0)]
&\geq
\frac{1}{n \Delta_n} \sum_{k=1}^n \exp(-2\theta_0^T Z_{t_{k-1}^n})|X_{t_k^n}-X_{t_{k-1}^n}|^2 
(\nu h^T Z_{t_{k-1}^n})^2 \\
&=\frac{\nu}{2} h^T V_n(0;\theta_0) h.
\end{align*}
We thus obtain the conclusion.
\qed
\end{proof}
\subsection{Proofs of main results}
Now, we are ready to prove our main results. 
\begin{proof}[Proof of Theorem \ref{re}]
It is sufficient to prove that $\|\psi_n(0;\theta_0)\|_\infty \leq \gamma_n$ implies that 
\[\|\hat{\theta}_n - \theta_0\|^2_2 \leq \frac{K_2 \gamma_n + K_3 \epsilon_n}{RE^2(T_0;J_n)}. \]
By the construction of the estimator $\hat{\theta}_n$, we have $\|\psi_n(0;\hat{\theta}_n)\|_\infty \leq \gamma_n$, which implies 
that 
\[\|\psi_n(0;\hat{\theta}_n) - \psi_n(0;\theta_0)\|_\infty \leq \|\psi_n(0;\hat{\theta}_n)\|_\infty + 
\|\psi_n(0;\theta_0)\|_\infty \leq 2\gamma_n . \]
Put $h := \theta_0 - \hat{\theta}_n$, then we have that $h \in C_{T_0}$ since it holds that
\begin{align*}
0 \geq \|\theta_0 - h\|_1 - \|\theta_0\|_1 &= \sum \limits_{j \in T_0^c} |h_{T^c_{0j}}| + \sum \limits_{j \in T_0}
 (|\theta_{0j}-h_{T_{0j}}| - |\theta_{0j}|)\\
 &\geq \sum \limits_{j \in T_0^c} |h_{T_{0j}^c}| -  \sum \limits_{j \in T_0} |h_{T_{0j}}| \\
 &= \|h_{T_0^c}\|_1 - \|h_{T_0}\|_1.
\end{align*}
Notice moreover that $\|h\|_1 \leq \|\hat{\theta}_n\|_1 + \|\theta_0\|_1 \leq 2 \|\theta_0\|_1$ by
the definition of $\hat{\theta}_n$.
Now, we use Lemma $4.3$ for $h$ to deduce that
\begin{align*}
h^T V_n(0;\theta_0)h 
&\leq \frac{2}{\nu} h^T [\psi_n(0;\hat{\theta}_n) - \psi_n(0;\theta_0)] \\
&\leq \frac{4}{\nu}\gamma_n \|h\|_1 \\
&\leq \frac{8}{\nu}\gamma_n \|\theta_0\|_1 \\
&=: K_2 \gamma_n.
\end{align*}
Thus it holds that
\begin{align*}
h^T J_n h 
&\leq |h^T (J_n - V_n(0;\theta_0))h| + h^T V_n(0;\theta_0) h \\
&\leq \epsilon_n \|h\|_1^2 + K_2 \gamma_n \\
&\leq \epsilon_n \cdot 4 \|\theta_0\|_1^2 + K_2 \gamma_n\\
&=:K_3 \epsilon_n + K_2 \gamma_n.
\end{align*}
By the definition of the restricted eigenvalue, we have that 
\begin{align*}
RE^2(T_0;J_n)
&\leq \frac{h^T J_n h}{\|\hat{\theta}_n-\theta_0\|_2^2} \\
&\leq \frac{K_2 \gamma_n + K_3 \epsilon_n}{\|\hat{\theta}_n-\theta_0\|_2^2}.
\end{align*}
Noting that $RE^2(T_0;J_n) >0$ with large probability when $n$ is sufficiently large, 
we obtain that 
\[\|\hat{\theta}_n - \theta_0\|_2^2 \leq \frac{K_2 \gamma_n + K_3 \epsilon_n}{RE^2(T_0;J_n)},\]
which yields the conclusion in (i). 
Using the factor $F_\infty(T_0;J_n)$, we obtain the conclusion in (ii) by the similar way.
\qed
\end{proof}
\begin{proof}[Proof of Theorem \ref{comp}]
It follows from the proof of Theorem \ref{re} that 
\[h^T V_n(0;\theta_0) h \leq K_4 \gamma_n \|\hat{\theta}_n-\theta_0\|_1.\]
Noting that $\|b\|^2_2 \leq \|b\|^2_1$ for all $b \in \mathbb{R}^{p_n}$, we have that
\[
h^T J_n h
\leq \epsilon_n \|\hat{\theta}_n-\theta_0\|^2_1 + K_4 \gamma_n \|\hat{\theta}_n - \theta_0\|_1.
\]
The definition of $\kappa(T_0;J_n)$ implies that 
\begin{align*}
\kappa^2(T_0;J_n)
&\leq \frac{S h^T J_nh}{\|h_{T_0}\|^2_1} \\
&\leq \frac{S \epsilon_n \|h\|^2_1 + K_4 S \gamma_n \|h\|_1}{\|h_{T_0}\|^2_1}.
\end{align*}
Since $\|h\|_1 \leq 2 \|h_{T_0}\|_1$, this yields the conclusion in (i). 

On the other hand, using the weak cone invertibility factor for every $q \geq 1$, we have that 
\[F_q (T_0;J_n) 
\leq  \frac{S^{\frac{1}{q}}\epsilon_n \|h\|^2_1 + S^{\frac{1}{q}}K_4 \gamma_n \|h\|_1}
{\|h_{T_0}\|_1 \|h\|_q},
\]
which implies that
\[
\|\hat{\theta}_n - \theta_0\|_q \leq
\frac{2S^{\frac{1}{q}}\epsilon_n \|\hat{\theta}_n-\theta_0\|_1 + 2S^{\frac{1}{q}}K_4 \gamma_n}{F_q(T_0;J_n)}.
\]
Using the $l_1$ bound derived above, we obtain the conclusion in (ii).
\qed
\end{proof}

\vskip 20pt
{\bf Acknowledgements.}
The second author's work was supported by Grant-in-Aid for Scientific Research (C), 15K00062, from Japan Society for the Promotion of Science.

\end{document}